\documentclass[preprint,12pt]{elsarticle}

\usepackage{amssymb}
\usepackage{amsthm}

\usepackage{amsmath}

\usepackage{dsfont}

\usepackage[T1]{fontenc}

\newtheorem{theorem}{Theorem}[section]
\newtheorem{lemma}[theorem]{Lemma}

\newtheorem{corollary}[theorem]{Corollary}

\theoremstyle{definition}

\newtheorem{example}[theorem]{Example}

\theoremstyle{remark}
\newtheorem{remark}[theorem]{Remark}

\makeatletter
\def\@author#1{\g@addto@macro\elsauthors{\normalsize%
    \def\baselinestretch{1}%
    \upshape\authorsep#1\unskip\textsuperscript{%
      \ifx\@fnmark\@empty\else\unskip\sep\@fnmark\let\sep=,\fi
      \ifx\@corref\@empty\else\unskip\sep\@corref\let\sep=,\fi
      }%
    \def\authorsep{\unskip,\space}%
    \global\let\@fnmark\@empty
    \global\let\@corref\@empty  
    \global\let\sep\@empty}%
    \@eadauthor={#1}
}
\makeatother

\begin{document}

\begin{frontmatter}



\title{A probabilistic generalization of the Stirling numbers of the second kind}


\author{Jos\'{e} A. Adell\corref{a1}\fnref{thanks}}
\ead{adell@unizar.es}

\author{Alberto Lekuona\fnref{thanks}}
\ead{lekuona@unizar.es}

\cortext[a1]{Corresponding author.}
\fntext[thanks]{The authors are partially supported by Research Projects DGA (E-64), MTM2015-67006-P, and by FEDER funds.}

\address{Departamento de M\'{e}todos Estad\'{\i}sticos, Facultad de
Ciencias, Universidad de Zaragoza, 50009 Zaragoza (Spain)}

\begin{abstract}
Associated to each random variable $Y$ having a finite moment generating function, we introduce a different generalization of the Stirling numbers of the second kind. Some characterizations and specific examples of such generalized numbers are provided. As far as their applications are concerned, attention is focused in extending in various ways the classical formula for sums of powers on arithmetic progressions. Illustrations involving rising factorials, Bell polynomials, polylogarithms, and a certain class of Appell polynomials, in connection with appropriate random variables $Y$ in each case, are discussed in detail.
\end{abstract}


\end{frontmatter}


\section{Introduction}\label{s1}

Let $\mathds{N}$ be the set of positive integers and $\mathds{N}_0=\mathds{N}\cup \{0\}$. Unless otherwise specified, we assume throughout this paper that $f:\mathds{R}\to \mathds{R}$ is an arbitrary function, $m,n,N\in \mathds{N}_0$, and $x\in \mathds{R}$. We also denote by $I_n(x)=x^n$ the monomial function.

The celebrated formula for sums of powers on arithmetic progressions states that
\begin{equation}\label{eq1.1}
    \sum_{k=0}^N I_n(x+k)=\dfrac{B_{n+1}(x+N+1)-B_{n+1}(x)}{n+1},
\end{equation}
where $B_n(x)$ is the $n$th Bernoulli polynomial. Since the time of James Bernoulli (1655--1705), different generalizations of such sums have been obtained (see, for instance, Kannappan and Zhang \cite{KanZha2002}, Guo and Zeng \cite{GuoZen2005}, Adell and Lekuona \cite{AdeLek2017}, Kim and Kim \cite{KimKim2017}, and the references therein).

On the other hand, the usual $m$th forward difference of $f$ is defined as
\begin{equation}\label{eq1.2}
    \Delta^m f(x)=\sum_{k=0}^m \binom{m}{k}(-1)^{m-k} f(x+k).
\end{equation}
From our point of view, an interesting fact is that the sums in \eqref{eq1.1} can also be computed in terms of forward differences of the monomial function $I_n(x)$. Actually, we have (see, for instance, Rosen \cite[p. 199]{Ros2000} or Spivey \cite{Spi2007})
\begin{equation}\label{eq1.3}
    \sum_{k=0}^N I_n(x+k)= \sum_{m=0}^{n\wedge N} \binom{N+1}{m+1} \Delta^m I_n(x),
\end{equation}
where $n\wedge N=\min (n,N)$. Computationally, formulas \eqref{eq1.1} and \eqref{eq1.3} are equivalent in the sense that the computation of a sum of $N+1$ terms is reduced to the computation of a polynomial in $N$ of degree $n+1$.

Finally, denote by
\begin{equation}\label{eq1.4}
    S(n,m;x):=\dfrac{\Delta^m I_n(x)}{m!},\quad m\leq n,
\end{equation}
the Stirling polynomials of the second kind, so that
\begin{equation}\label{eq1.5}
    S(n,m):=S(n,m;0),\quad m\leq n,
\end{equation}
are the classical Stirling numbers of the second kind (see Abramowitz and Stegun \cite{AbrSte1992} or Roman \cite[p. 60]{Rom1984} for other equivalent definitions). Obviously, formula \eqref{eq1.3} can be rewritten in terms of $S(n,m;x)$.

In this paper, we consider the following probabilistic generalization of the polynomials $S(n,m;x)$. Suppose that $Y$ is a random variable having a finite moment generating function, i.e.,
\begin{equation}\label{eq1.6}
    \mathds{E}e^{r|Y|}<\infty,
\end{equation}
for some $r>0$, where $\mathds{E}$ stands for mathematical expectation. Let $(Y_j)_{j\geq 1}$ be a sequence of independent copies of $Y$ and denote by
\begin{equation}\label{eq1.7}
    S_k=Y_1+\cdots +Y_k,\quad k\in \mathds{N}\qquad (S_0=0).
\end{equation}
We define the Stirling polynomials of the second kind associated to $Y$ as
\begin{equation}\label{eq1.8}
    S_Y(n,m;x)=\dfrac{1}{m!}\sum_{k=0}^m \binom{m}{k}(-1)^{m-k}\mathds{E}(x+S_k)^n,\quad m\leq n,
\end{equation}
as well as the Stirling numbers of the second kind associated to $Y$ as
\begin{equation}\label{eq1.9}
    S_Y(n,m)=S_Y(n,m;0),\quad m\leq n.
\end{equation}
Note that if $Y=1$, we have from \eqref{eq1.4}, \eqref{eq1.5}, and \eqref{eq1.8}
\begin{equation*}
    S_1(n,m;x)=S(n,m;x),\qquad S_1(n,m)=S(n,m).
\end{equation*}

Motivated by various specific problems, different generalizations of the Stirling numbers $S(n,m)$ have been considered in the literature (see, for instance, Comtet \cite{Com1972}, Hsu and Shine \cite{HsuShi1998}, Mez\H{o} \cite{Mez2010}, Luo and Srivastava \cite{LuoSri2011}, Mihoubi and Maamra \cite{MihMaa2012}, Wuyungaowa \cite{Wuy2013}, Caki\'{c} \textit{et al.} \cite{CakElDMil2013}, Mihoubi and Tiachachat \cite{MihTia2014}, and El-Desouky \textit{et al.} \cite{ElDCakShi2017}, among many others). As far as the applications of the polynomials $S_Y(n,m;x)$ are concerned, attention is focused in showing the following extension of formulas \eqref{eq1.3} and \eqref{eq1.4}
\begin{equation}\label{eq1.9.2}
    \sum_{k=0}^N \mathds{E}I_n(x+S_k)=\sum_{m=0}^{n\wedge N}\binom{N+1}{m+1}m! S_Y(n,m;x),
\end{equation}
and its consequences. The versatility of formula \eqref{eq1.9.2} comes from the fact that each choice of the random variable $Y$ gives us a different (sometimes unexpected) summation formula. This is illustrated in Section~\ref{s4}, where we obtain summation formulas involving rising factorials, Bell polynomials, and polylogarithms by considering random variables $Y$ having the exponential, the Poisson, and the geometric distributions, respectively. Analogous formulas are obtained for a certain class of Appell polynomials by considering the normal, the logistic and the hyperbolic secant distributions, among others.

In Section~\ref{s3}, we give different characterizations of the generalized polynomials $S_Y(n,m;x)$, as well as some extensions of formula \eqref{eq1.9.2}, whereas specific examples of the generalized Stirling numbers $S_Y(n,m)$ are presented in Section~\ref{s4}. To clarify the proof of such results, we have included in Section~\ref{s2} a reformulation of equality \eqref{eq1.3}.

\section{The classical formula for sums of powers}\label{s2}

Our approach to this classical formula is based on the following two simple properties. The first one (see, for instance, Flajolet and Vepstas \cite{FlaVep2008}, Mu \cite{Mu2013} or Adell and Lekuona \cite{AdeLek2017}) says that
\begin{equation}\label{eq2.10}
    f(x+k)=\sum_{m=0}^k \binom{k}{m}\Delta^m f(x),\quad k\in \mathds{N}_0,
\end{equation}
whereas the second one states that
\begin{equation}\label{eq2.11}
    \Delta^m p_n(x)=0,\quad m=n+1,n+2,\ldots ,
\end{equation}
for any polynomial $p_n(x)$ of exact degree $n$, as follows easily from \eqref{eq1.2}.

\begin{theorem}\label{th1}
We have
\begin{equation}\label{eq2.12}
    \sum_{k=0}^N I_n(x+k)= \sum_{m=0}^{n\wedge N} \binom{N+1}{m+1} m! S(n,m;x)=\sum_{k=0}^{n\wedge N} c_{n,N}(k) I_n(x+k),
\end{equation}
where
\begin{equation}\label{eq2.13}
    c_{n,N}(k)=\sum_{m=k}^{n\wedge N}\binom{N+1}{m+1}\binom{m}{k}(-1)^{m-k},\quad k=0,1,\ldots ,n\wedge N.
\end{equation}

Moreover,
\begin{equation}\label{eq2.14}
    \sum_{k=0}^{n\wedge N}c_{n,N}(k)=N+1.
\end{equation}
\end{theorem}

\begin{proof}
Choosing $f(x)=I_n(x)$ in \eqref{eq2.10} and taking into account \eqref{eq1.4}, we have
\begin{equation}\label{eq2.15}
    \begin{split}
       & \sum_{k=0}^N I_n(x+k)= \sum_{k=0}^N \sum_{m=0}^k \binom{k}{m} m! S(n,m;x) \\
        & = \sum_{m=0}^N m! S(n,m;x) \sum_{k=m}^N \binom{k}{m}=\sum_{m=0}^{n\wedge N}\binom{N+1}{m+1}m! S(n,m;x),
    \end{split}
\end{equation}
where in the last equality we have used \eqref{eq1.4}, \eqref{eq2.11}, and the elementary combinatorial identity
\begin{equation}\label{eq2.15.2}
    \sum_{k=m}^N \binom{k}{m}=\binom{N+1}{m+1},\quad m\leq N.
\end{equation}
Again by \eqref{eq1.4}, the last term in \eqref{eq2.15} equals to
\begin{equation*}
    \sum_{m=0}^{n\wedge N}\binom{N+1}{m+1}\sum_{k=0}^m \binom{m}{k}(-1)^{m-k}I_n(x+k)=\sum_{k=0}^{n\wedge N} c_{n,N}(k) I_n(x+k),
\end{equation*}
where $c_{n,N}(k)$ is defined in \eqref{eq2.13}. Finally, we have
\begin{equation*}
    \begin{split}
        & \sum_{k=0}^{n\wedge N}c_{n,N}(k)=\sum_{k=0}^{n\wedge N} \sum_{l=0}^{n\wedge N-k}\binom{N+1}{k+l+1}\binom{k+l}{k}(-1)^l \\
         & = \sum_{s=0}^{n\wedge N}\sum_{k+l=s}\binom{N+1}{k+l+1}\binom{k+l}{k}(-1)^l\\
         &=\sum_{s=0}^{n\wedge N} \binom{N+1}{s+1}\sum_{k=0}^s \binom{s}{k}(-1)^{s-k}=N+1.
     \end{split}
\end{equation*}
This shows \eqref{eq2.14} and completes the proof.
\end{proof}

\begin{remark}
If $N\leq n$, then $c_{n,N}(k)=1$, $k=0,1,\ldots ,N$, and therefore the last equality in \eqref{eq2.12} is trivial. If $N>n$, then the finite sequence $c_{n,N}(k)$, $k=0,1,\ldots ,n$ is an alternating sequence, since it can be checked that
\begin{equation*}
    c_{n,N}(k)=1+(-1)^{n-k}\sum_{i=0}^{N-(n+1)}\binom{n+1+i}{k}\binom{n-k+i}{n-k},\quad k=0,1,\ldots, n.
\end{equation*}
Details are omitted.
\end{remark}

\begin{theorem}\label{th1.1}
Let $p_n(x)$ be a polynomial of exact degree $n$. Then,
\begin{equation}\label{eq2.15.3}
    \sum_{k=0}^N p_n(x+k)=\sum_{m=0}^{n\wedge N}\binom{N+1}{m+1} \Delta^m p_n(x)=\sum_{k=0}^{n\wedge N} c_{n,N}(k) p_n(x+k),
\end{equation}
where the coefficients $c_{n,N}(k)$ are defined in \eqref{eq2.13}.
\end{theorem}

\begin{proof}
The proof is the same as that of Theorem~\ref{th1} replacing the monomial $I_n(x)$ by the polynomial $p_n(x)$ and applying formulas \eqref{eq2.10} and \eqref{eq2.11}.
\end{proof}

We emphasize that the last two expressions in \eqref{eq2.15.3} are computationally equivalent. The last one should be used when the forward differences $\Delta^m p_n(x)$ are difficult to evaluate. In this regard, note that the coefficients $c_{n,N}(k)$ can be easily computed and stored once and for all in order to use them in any specific example (see Section~\ref{s4} and particularly Theorem~\ref{th12}).

\section{The polynomials $S_Y(n,m;x)$}\label{s3}

In what follows, we denote by $(U_j)_{j\geq 1}$ a sequence of independent identically distributed random varaibles having the uniform distribution on $[0,1]$, and assume that $(U_j)_{j\geq 1}$ and $(Y_j)_{j\geq 1}$, as given in \eqref{eq1.7}, are mutually independent. We consider the difference operator
\begin{equation*}
    \Delta_y^1 f(x)=f(x+y)-f(x),\quad y\in \mathds{R},
\end{equation*}
together with the iterates
\begin{equation}\label{eq3.16}
    \Delta_{y_1,\ldots, y_m}f(x)=(\Delta_{y_1}^1\circ \cdots \circ \Delta_{y_m}^1)f(x),\quad (y_1,\ldots, y_m)\in \mathds{R}^m,\quad m\in \mathds{N}.
\end{equation}
It seems that such iterates have been independently introduced by Mrowiec \textit{et al.} \cite{MroRajWas2017} in connection with Wright-convex functions of order $m$ and by Dilcher and Vignat \cite{DilVig2016} in the context of convolution identities for Bernoulli polynomials (see also \cite{AdeLek2017.2}). For any $m\in \mathds{N}$, denote by
\begin{equation*}
    I_m(k)=\left \{(i_1,\ldots ,i_k)\in \{1,\ldots ,m\}:\ i_r\neq i_s\ \textrm{if}\ r\neq s \right \},\quad k=1,\ldots ,m.
\end{equation*}
The following lemma has been shown in \cite{AdeLek2017.2} (see also \cite[Section~2]{AdeLek2018}).

\begin{lemma}\label{l4}
Let $m\in \mathds{N}$ and $(y_1,\ldots ,y_m)\in \mathds{R}^m$. Then,
\begin{equation}\label{eq3.17}
    \Delta_{y_1,\ldots ,y_m}^m f(x)=(-1)^mf(x)+\sum_{k=1}^m (-1)^{m-k} \sum_{I_m(k)} f(x+y_{i_1}+\cdots +y_{i_k}).
\end{equation}
If, in addition, $f$ is $m$ times differentiable, then
\begin{equation}\label{eq3.18}
    \Delta_{y_1,\ldots, y_m}^m f(x)=y_1\cdots y_m \mathds{E}f^{(m)}(x+y_1U_1+\cdots +y_mU_m).
\end{equation}
\end{lemma}

Setting $y_1=\cdots =y_m=1$ in \eqref{eq3.17}, we obtain the usual $m$th forward difference of $f$, that is,
\begin{equation}\label{eq3.19}
    \Delta_{1,\ldots ,1}^m f(x)=\Delta^m f(x),\quad m\in \mathds{N}.
\end{equation}

Let $m\in \mathds{N}$ and $(y_1,\ldots ,y_m)\in \mathds{R}^m$. If $p_n(x)$ is a polynomial of exact degree $n$, we have from \eqref{eq3.18}
\begin{equation}\label{eq3.20}
    \Delta_{y_1,\ldots ,y_m}^m p_n(x)=0,\quad m=n+1,n+2,\ldots
\end{equation}
Also, for the exponential function $e_z(x)=e^{zx}$, $z\in \mathds{C}$, we have from Lemma~\ref{l4}
\begin{equation}\label{eq3.21}
\begin{split}
    &\Delta_{y_1,\ldots ,y_m}^m e_z(x)= e^{zx}(e^{zy_1}-1)\cdots (e^{zy_m}-1)\\
    &=y_1\ldots y_m z^m \mathds{E}e^{z(x+y_1U_1+\cdots +y_mU_m)}.
\end{split}
\end{equation}

In order to write the formulas below in a unified way for any $m\in \mathds{N}_0$, we make use of the conventions
\begin{equation}\label{eq3.22}
    \Delta_\phi^0 f(x)=f(x),\quad \sum_{k=1}^0 x_k=0,\quad \prod_{k=1}^0 x_k =1.
\end{equation}

Despite the cumbersome expression in \eqref{eq3.17}, if we replace each $y_j$ by the random variable $Y_j$, as given in \eqref{eq1.7}, and then take expectations, we obtain the following auxiliary result already shown in \cite[Lemma~2.1]{AdeLek2018}.

\begin{lemma}\label{l5}
For any $m\in \mathds{N}_0$, we have
\begin{equation}\label{eq3.23}
    \mathds{E}\Delta_{Y_1,\ldots, Y_m}^m f(x)= \sum_{k=0}^m \binom{m}{k}(-1)^{m-k} \mathds{E}f(x+S_k).
\end{equation}
If, in addition, $f$ is $m$ times differentiable, then
\begin{equation}\label{eq3.24}
    \mathds{E}\Delta_{Y_1,\ldots ,Y_m}^m f(x)= \mathds{E}Y_1 \cdots Y_m f^{(m)}(x+Y_1U_1+\cdots +Y_mU_m).
\end{equation}
\end{lemma}

Choosing $f(x)=I_n(x)$ in \eqref{eq3.23} and recalling \eqref{eq3.20}, we see that the definition of the Stirling polynomials of the second kind associated to $Y$, given in \eqref{eq1.8}, makes sense. In other words,
\begin{equation}\label{eq3.25}
    S_Y(n,m;x)=0,\quad m=n+1,n+2,\ldots
\end{equation}

The following result, which generalizes in various ways Proposition~2.2 in \cite{AdeLek2018}, gives us equivalent definitions of the polynomials $S_Y(n,m;x)$. In this respect, recall that the classical Stirling numbers of the first and second kind, respectively denoted by $s(n,k)$ and $S(n,k)$, $k=0,1,\ldots ,n$, are also defined by
\begin{equation}\label{eq3.26}
    (x)_n=\sum_{k=0}^n s(n,k)x^k,\quad x^n=\sum_{k=0}^n S(n,k) (x)_k,
\end{equation}
where $(x)_n=x(x-1)\cdots (x-n+1)$ is the descending factorial.

\begin{theorem}\label{th6}
For any $m,n\in \mathds{N}_0$ with $m\leq n$, we have
\begin{equation}\label{eq3.27}
    \begin{split}
      S_Y(n,m;x)& =\dfrac{1}{m!} \mathds{E} \Delta_{Y_1,\ldots ,Y_m}^m I_n(x)\\
      & = \binom{n}{m} \mathds{E}Y_1\cdots Y_m I_{n-m} (x+Y_1U_1+\cdots +Y_mU_m) \\
        & =\dfrac{1}{m!}\sum_{i=0}^n S(n,i) \sum_{k=0}^m \binom{m}{k}(-1)^{m-k}\mathds{E}(x+S_k)_i.
    \end{split}
\end{equation}

Equivalently, the polynomials $S_Y(n,m;x)$ are defined via their generating function as
\begin{equation}\label{eq3.28}
    \dfrac{e^{zx}}{m!}\left ( \mathds{E}e^{zY}-1\right )^m=\sum_{n=m}^\infty \dfrac{S_Y(n,m;x)}{n!}z^n,\quad z\in \mathds{C},\quad |z|\leq r.
\end{equation}
\end{theorem}

\begin{proof}
The case $m=0$ being trivial by convention \eqref{eq3.22}, assume that $m\in \mathds{N}$. The first two equalities in \eqref{eq3.27} readily follow from \eqref{eq1.8} and Lemma~\ref{l5}, by choosing $f(x)=I_n(x)$, whereas the third one is an immediate consequence of \eqref{eq1.8} and \eqref{eq3.26}.

To show \eqref{eq3.28}, replace $y_j$ by $Y_j$, $j=1,\ldots ,m$ in the first equality in \eqref{eq3.21} and then take expectations. Using the independence and the identical distribution of the random variables involved and the linearity of the generalized difference operator, we obtain
\begin{equation*}
    \dfrac{e^{zx}}{m!} (\mathds{E}e^{zY}-1)^m=\dfrac{1}{m!} \mathds{E}\Delta_{Y_1,\ldots ,Y_m}^m e_z(x)=\sum_{n=m}^\infty \dfrac{z^n}{n!}\dfrac{1}{m!}\mathds{E}\Delta_{Y_1,\ldots ,Y_m}^m I_n(x),
\end{equation*}
where in the last equality we have used \eqref{eq3.20}. This, together, with \eqref{eq3.27}, shows \eqref{eq3.28} and completes the proof.
\end{proof}

Choosing $Y=1$ in the second equality in \eqref{eq3.27}, we obtain the following probabilistic representation for the Stirling numbers of the second kind
\begin{equation}\label{eq3.28.2}
    S(n,m)= \binom{n}{m}\mathds{E}(U_1+\cdots +U_m)^{n-m}.
\end{equation}
This representation was already proved by Sun \cite{Sun2005}. On the other hand, the last equality in \eqref{eq3.27} is very useful to obtain explicit expressions for $S_Y(n,m)$ when $Y$ is a random variable taking values in $\mathds{N}_0$.

We are in a position to state the following result which extends Theorem~\ref{th1.1}.

\begin{theorem}\label{th7}
Let $p_n(x)$ be a polynomial of exact degree $n$. Then,
\begin{equation}\label{eq3.29}
\begin{split}
    &\sum_{k=0}^N \mathds{E}p_n(x+S_k)= \sum_{m=0}^{n\wedge N}\binom{N+1}{m+1}\mathds{E} \Delta_{Y_1,\ldots ,Y_m}^m p_n(x)\\
    &=\sum_{k=0}^{n\wedge N} c_{n,N}(k) \mathds{E}p_n(x+S_k),
\end{split}
\end{equation}

where $c_{n,N}(k)$ is defined in \eqref{eq2.13}.
\end{theorem}

\begin{proof}
Fix $n\in \mathds{N}_0$ and $x\in \mathds{R}$ and consider the function
\begin{equation*}
    g_{n,x}(k)=\mathds{E}p_n(x+S_k),\quad k\in \mathds{N}_0.
\end{equation*}
In view of \eqref{eq1.2}, formula \eqref{eq3.23} applied to the function $f(x)=p_n(x)$ tells us that
\begin{equation*}
    \Delta^m g_{n,x}(0)= \mathds{E}\Delta_{Y_1,\ldots ,Y_m}^m p_n(x).
\end{equation*}
This implies, by virtue of \eqref{eq2.10}, that
\begin{equation*}
    \mathds{E}p_n(x+S_k)=\sum_{m=0}^k \binom{k}{m}\mathds{E}\Delta_{Y_1,\ldots ,Y_m}^m p_n(x).
\end{equation*}
We therefore have from \eqref{eq2.15.2} and \eqref{eq3.20}
\begin{equation*}
    \begin{split}
        & \sum_{k=0}^N \mathds{E}p_n(x+S_k)=\sum_{k=0}^N \sum_{m=0}^k \binom{k}{m} \mathds{E}\Delta_{Y_1,\ldots, Y_m}^m p_n(x) \\
         & =\sum_{m=0}^N \mathds{E}\Delta_{Y_1,\ldots ,Y_m}^m p_n(x) \sum_{k=m}^N \binom{k}{m}=\sum_{m=0}^{n\wedge N} \binom{N+1}{m+1}\mathds{E}\Delta_{Y_1,\ldots ,Y_m}^m p_n(x).
     \end{split}
\end{equation*}
Having in mind \eqref{eq3.23}, the proof of the second equality in \eqref{eq3.29} follows along the lines of that of the second equality in \eqref{eq2.12}. The proof is complete.
\end{proof}

As a consequence, we obtain the following extension of the classical formula for sums of powers on arithmetic progressions.

\begin{corollary}\label{c8}
We have
\begin{equation*}
    \sum_{k=0}^N \mathds{E}I_n(x+S_k)=\sum_{m=0}^{n\wedge N}\binom{N+1}{m+1}m! S_Y (n,m;x)=\sum_{k=0}^{n\wedge N}c_{n,N}(k) \mathds{E}I_n (x+S_k).
\end{equation*}
\end{corollary}

\begin{proof}
Apply Theorem~\ref{th7} with $p_n(x)=I_n(x)$ and the first equality in \eqref{eq3.27}.
\end{proof}

\section{Examples}\label{s4}

We keep here the same notations used in Sections~\ref{s1} and \ref{s2}, particularly, the random variables $Y$, $(Y_j)_{j\geq 1}$ and $(S_k)_{k\geq 0}$ given in \eqref{eq1.6} and \eqref{eq1.7}, and the coefficients $c_{n,N}(k)$ defined in \eqref{eq2.13}.

\paragraph{\textbf{a) Some specific examples}}
Each choice of the probability law of the random variable $Y$ gives us a different expression for $S_Y(n,m;x)$, $m\leq n$. We mention here the following concrete examples.

\begin{example}\label{ex4.1}
Let $\alpha\in \mathds{R}$. Recall (see, for instance, Mez\H{o} \cite{Mez2010}, Mihoubi and Tiachachat \cite{MihTia2014}, and El-Desouky \textit{et al.} \cite{ElDCakShi2017}) that the $x$-Whitney numbers of the second kind $W_\alpha (n,m;x)$ are defined via their generating function
\begin{equation}\label{eq4.1.a}
    \sum_{n=m}^\infty W_\alpha (n,m;x) \dfrac{z^n}{n!}= \dfrac{e^{zx}}{m!} \left (\dfrac{e^{\alpha z}-1}{\alpha}\right )^m.
\end{equation}
Choosing $Y=\alpha$ and comparing \eqref{eq3.28} with \eqref{eq4.1.a}, we immediately see that
\begin{equation*}
    S_Y(n,m;x)=\alpha^m W_\alpha (n,m;x),\quad m\leq n.
\end{equation*}
\end{example}

\begin{example}\label{ex4.2}
Let $0<p\leq 1$. Suppose that $Y$ has the Bernoulli law
\begin{equation*}
    P(Y=1)=p=1-P(Y=0).
\end{equation*}
Since
\begin{equation*}
    \mathds{E}e^{zY}-1=p(e^z-1),
\end{equation*}
we readily have from \eqref{eq3.28}
\begin{equation*}
    S_Y(n,m;x)=p^m S(n,m;x),\quad m\leq n.
\end{equation*}
\end{example}

\begin{example}\label{ex4.3}
Recall that a random variable $X$ has the normal distribution with mean $\mu \in \mathds{R}$ and variance $\sigma^2 >0$, denoted by $X\sim N(\mu, \sigma^2)$, if its probability density is given by
\begin{equation*}
    \rho(\theta)=\dfrac{1}{\sqrt{2\pi\sigma^2}} e^{-(\theta-\mu)^2/2\sigma^2},\quad \theta \in \mathds{R}.
\end{equation*}
Suppose that $Y\sim N(0,1)$. Denote by $H(x)=(H_n(x))_{n\geq 0}$ the sequence of Hermite polynomials. It is known that such polynomials can be written in probabilistic terms as (cf. Withers \cite{Wit2000}, Adell and Lekuona \cite{AdeLek2006}, or Ta \cite{Ta2015})
\begin{equation}\label{eq4.1.b}
    H_n(x)=\mathds{E}(x+iY)^n,
\end{equation}
where $i$ is the imaginary unit. It is also well known (cf. Johnson \textit{et al.} \cite[p. 91]{JohKotBal1994}) that $S_k\sim N(0,k)$ and, therefore, $S_k$ has the same law as $\sqrt{k}Y$. This implies, by virtue of \eqref{eq1.8} and \eqref{eq4.1.b}, that $S_Y(2n+1,m)=0$, $m\leq 2n+1$, as well as
\begin{equation*}
    S_Y(2n,m)=(-1)^nH_{2n}(0)S(n,m),\quad m\leq n.
\end{equation*}
\end{example}

\begin{example}\label{ex4.4}
Suppose that $Y$ has the uniform distribution on $[0,1]$. It follows from the probabilistic representation given in \eqref{eq3.28.2} that
\begin{equation*}
    \mathds{E}S_k^n=\dfrac{S(n+k,k)}{\binom{n+k}{k}}.
\end{equation*}
We therefore have from \eqref{eq1.8}
\begin{equation*}
    S_Y(n,m)=\dfrac{n!}{(n+m)!} \sum_{k=0}^m \binom{n+m}{n+k}(-1)^{m-k}S(n+k,k),\quad m\leq n.
\end{equation*}
\end{example}

\begin{example}\label{ex4.5}
Assume that $Y=UT$, where $U$ and $T$ are two independent random variables such that $U$ is uniformly distributed on $[0,1]$ and $T$ has the exponential density $\rho_1(\theta)=e^{-\theta}$, $\theta>0$. In \cite{AdeLek2017.3}, we have shown the following probabilistic representation for the Stirling numbers of the first kind $s(n,k)$
\begin{equation*}
    s(n,k)=(-1)^{n-k} \binom{n}{k}\mathds{E}S_k^{n-k},\quad k=0,1,\ldots ,n.
\end{equation*}
Thus, we have from \eqref{eq1.8}
\begin{equation*}
    S_Y(n,m)=(-1)^n \dfrac{n!}{(n+m)!} \sum_{k=0}^m \binom{n+m}{n+k}(-1)^{m-k} s(n+k,k),\quad m\leq n.
\end{equation*}
\end{example}

\paragraph{\textbf{b) Rising factorials and the exponential distribution}}

Denote by
\begin{equation*}
    \langle x \rangle _n= x(x+1)\cdots (x+n-1),\quad n\in \mathds{N},\quad \langle x \rangle _0=1,
\end{equation*}
the rising factorial. For fixed $n$, it is easily seen by induction on $m$ that
\begin{equation}\label{eq4.30}
    \Delta^m \langle x \rangle _n= \dfrac{n!}{(n-m)!} \langle x+m \rangle _{n-m},\quad m=0,1,\ldots ,n.
\end{equation}
For any $\alpha >0$, consider the following probability gamma density
\begin{equation}\label{eq4.31}
    \rho_{\alpha}(\theta)=\dfrac{\theta^{\alpha-1}}{\Gamma (\alpha)}e^{-\theta},\quad \theta >0.
\end{equation}

\begin{theorem}\label{th9}
Let $Y$ be a random variable having the exponential density $\rho_1 (\theta)$. Then,
\begin{equation}\label{eq4.32}
    S_Y(n,m)=\binom{n}{m}\langle m \rangle _{n-m},\quad m\leq n.
\end{equation}

As a consequence, we have for any $N\geq n$
\begin{equation}\label{eq4.33}
    \sum_{k=0}^N \langle k \rangle _n= \sum_{m=0}^n \binom{N+1}{m+1}(n)_m \langle m \rangle _{n-m}=\sum_{k=0}^n c_{n,N}(k) \langle k \rangle _n.
\end{equation}
\end{theorem}

\begin{proof}
Let $k\in \mathds{N}$. It is well known (cf. Johnson \textit{et al.} \cite[pp. 337--340]{JohKotBal1994}) that the random variable $S_k$ has the gamma density $\rho_k (\theta)$ defined in \eqref{eq4.31}, thus having
\begin{equation}\label{eq4.34}
    \mathds{E}S_k^n=\int_0^\infty \theta^n \rho_k(\theta)\, d\theta = \langle k \rangle _n.
\end{equation}
We therefore have from definition \eqref{eq1.8} and \eqref{eq4.30}
\begin{equation*}
    S_Y(n,m)=\dfrac{1}{m!}\sum_{k=0}^m \binom{m}{k}(-1)^{m-k}\langle k \rangle _n=\dfrac{1}{m!}\Delta^m \langle 0 \rangle _n=\binom{n}{m}\langle m \rangle _{n-m},
\end{equation*}
thus showing \eqref{eq4.32}. Formula \eqref{eq4.33} readily follows from Corollary~\ref{c8}, \eqref{eq4.32}, and \eqref{eq4.34}. The proof is complete.
\end{proof}

\paragraph{\textbf{c) Bell polynomials and the Poisson distribution}}

The Bell polynomials are defined by (see, for instance, Roman \cite[pp. 63--67]{Rom1984})
\begin{equation}\label{eq4.35}
    B_n(x)=\sum_{j=0}^\infty j^n \dfrac{x^j}{j!}e^{-x}=\sum_{j=0}^n S(n,j)x^j,
\end{equation}
where the second equality in \eqref{eq4.35} is known as Dobi\'nski's formula. On the other hand, fix $\lambda \geq 0$ and let $Y$ be a random variable having the Poisson law with mean $\lambda$, i.e.,
\begin{equation}\label{eq4.36}
    P(Y=j)= \dfrac{\lambda ^j}{j!}e^{-\lambda},\quad j\in \mathds{N}_0.
\end{equation}
Observe that
\begin{equation}\label{eq4.37}
    B_n(\lambda)=\mathds{E}Y^n,\quad \lambda \geq 0.
\end{equation}

\begin{theorem}\label{th10}
Let $Y$ be as in \eqref{eq4.36}. Then,
\begin{equation}\label{eq4.38}
    S_Y (n,m)= \sum_{r=m}^n S(n,r)S(r,m)\lambda^r,\quad m\leq n.
\end{equation}
As a consequence, we have for any $N\geq n$
\begin{equation}\label{eq4.39}
    \sum_{k=0}^N B_n(k\lambda)=\sum_{m=0}^n \binom{N+1}{m+1} m! S_Y(n,m)=\sum_{k=0}^n c_{n,N}(k) B_n(k\lambda).
\end{equation}
\end{theorem}

\begin{proof}
Clearly, we have from \eqref{eq4.36}
\begin{equation*}
    \mathds{E}e^{zY}=e^{\lambda (e^z-1)},\quad z\in \mathds{C}.
\end{equation*}
Applying \eqref{eq3.28} to the classical Stirling numbers, we thus have
\begin{equation*}
    \begin{split}
        & \dfrac{1}{m!}(\mathds{E}e^{zY}-1)^m= \dfrac{1}{m!} \left ( e^{\lambda (e^z-1)}-1\right )^m= \sum_{r=m}^\infty S(r,m) \lambda ^r \dfrac{(e^z-1)^r}{r!} \\
         & =\sum_{r=m}^\infty S(r,m)\lambda^r \sum_{n=r}^\infty \dfrac{S(n,r)}{n!}z^n= \sum_{m=n}^\infty \dfrac{z^n}{n!}\sum_{r=m}^n S(n,r)S(r,m)\lambda^r,
     \end{split}
\end{equation*}
which, by virtue of \eqref{eq3.28}, shows \eqref{eq4.38}. On the other hand, it is well known (cf. Johnson \textit{et al.} \cite[p. 160]{JohKemKot2005}) that $S_k$ has the Poisson law with mean $k\lambda$. By \eqref{eq4.37}, this implies that
\begin{equation*}
    \mathds{E}S_k^n=B_n(k\lambda).
\end{equation*}
Therefore, \eqref{eq4.39} follows by choosing $x=0$ in Corollary~\ref{c8}.
\end{proof}

Formula \eqref{eq4.39} is still valid if we replace $\lambda \geq 0$ by $x\in \mathds{R}$. The reason is that each term in \eqref{eq4.39} is a polynomial in $\lambda$ of degree $n$.

\paragraph{\textbf{d) Polylogarithms and the geometric distribution}}

Recall that the polylogarithm function is defined as
\begin{equation}\label{eq4.40}
    Li_s(z)=\sum_{j=1}^\infty \dfrac{z^j}{j^s},\quad z\in \mathds{C},\quad |z|<1,\quad s\in \mathds{C}.
\end{equation}
On the other hand, fix $0<q<1$ and set $p=1-q$. Let $Y$ be a random variable having the geometric distribution
\begin{equation}\label{eq4.41}
    P(Y=j)=pq^j,\quad j\in \mathds{N}_0.
\end{equation}
Note that
\begin{equation}\label{eq4.42}
    Li_{-n}(q)=\dfrac{q}{p}\sum_{j=0}^\infty (j+1)^n pq^j=\dfrac{q}{p}\mathds{E}(Y+1)^n.
\end{equation}
Here, we will be concerned with multinomial convolutions of the polylogarithm function defined as
\begin{equation}\label{eq4.43}
    Li_{-n}^{\star k}(q)=\sum_{n_1+\cdots + n_k=n} \dfrac{n!}{n_1!\cdots n_k!} Li_{-n_1}(q)\cdots Li_{-n_k}(q),\quad k\in \mathds{N},
\end{equation}
and $L_{i_{-n}}^{\star 0}(q)=\delta_{n0}$. Using \eqref{eq4.42}, we can rewrite \eqref{eq4.43} in probabilistic terms as
\begin{equation}\label{eq4.44}
    \begin{split}
      & Li_{-n}^{\star k}(q)=\left ( \dfrac{q}{p}\right )^k \sum_{n_1+\cdots +n_k=n} \dfrac{n!}{n_1!\cdots n_k!} \mathds{E}(Y_1+1)^{n_1}\cdots \mathds{E}(Y_k+1)^{n_k}\\
      & = \left ( \dfrac{q}{p}\right )^k \mathds{E}(Y_1+1+\cdots +Y_k+1)^n=\left ( \dfrac{q}{p}\right )^k \mathds{E}(S_k+k)^n,\quad k\in \mathds{N}_0.
   \end{split}
\end{equation}

With the preceding notations, we state the following.

\begin{theorem}\label{th11}
Let $Y$ be as in \eqref{eq4.41}. Then,
\begin{equation}\label{eq4.45}
    S_{Y+1}(n,m)=\dfrac{1}{q^m}\sum_{r=m}^n \binom{r}{m}\langle m \rangle_{r-m}S(n,r) \left (\dfrac{q}{p}\right )^r,\quad m\leq n.
\end{equation}
As a consequence, we have for any $N\geq n$
\begin{equation}\label{eq4.46}
\begin{split}
    &\sum_{k=0}^N \left ( \dfrac{p}{q}\right )^k Li_{-n}^{\star k}(q)=\sum_{m=0}^n \binom{N+1}{m+1}m!\, S_{Y+1}(n,m)\\
    &=\sum_{k=0}^n c_{n,N}(k) \left ( \dfrac{p}{q}\right )^k Li_{-n}^{\star k}(q).
\end{split}
\end{equation}
\end{theorem}

\begin{proof}
Let $z\in \mathds{C}$ with $|e^z-1|<p/q$. By \eqref{eq4.41}, we get
\begin{equation}\label{eq4.47}
    \mathds{E}e^{z(Y+1)}-1=\dfrac{pe^z}{1-qe^z}-1=\dfrac{1}{p}\, \dfrac{e^z-1}{1-\dfrac{q}{p}(e^z-1)}.
\end{equation}
Using the binomial expansion, we have from \eqref{eq3.28} and \eqref{eq4.47}
\begin{equation*}
    \begin{split}
        & \dfrac{1}{m!}\left ( \mathds{E}e^{z(Y+1)}-1\right )^m= \dfrac{1}{m!}\left ( \dfrac{e^z-1}{p} \right )^m\sum_{k=0}^\infty \binom{-m}{k}\left ( -\dfrac{q}{p}\right )^k (e^z-1)^k\\
         & =\dfrac{1}{p^m}\sum_{k=0}^\infty \binom{m+k}{m}\langle m\rangle_k \left ( \dfrac{q}{p}\right )^k \dfrac{(e^z-1)^{m+k}}{(m+k)!}\\
         &=\dfrac{1}{p^m}\sum_{r=m}^\infty \binom{r}{m}\langle m\rangle_{r-m}\left ( \dfrac{q}{p}\right )^{r-m} \sum_{n=r}^\infty \dfrac{S(n,r)}{n!}\, z^n\\
         &= \dfrac{1}{q^m}\sum_{n=m}^\infty \dfrac{z^n}{n!}\sum_{r=m}^n \binom{r}{m}\langle m\rangle_{r-m}\left ( \dfrac{q}{p}\right )^r S(n,r).
     \end{split}
\end{equation*}
Again by \eqref{eq3.28}, this shows formula \eqref{eq4.45}. Finally, expression \eqref{eq4.46} follows from Corollary~\ref{c8}, \eqref{eq4.44}, and \eqref{eq4.45}. The proof is complete.
\end{proof}

\paragraph{\textbf{e) Appell polynomials and moments of random variables}}

Denote by $\mathcal{G}$ the set of real sequences $u=(u_n)_{n\geq 0}$ such that $u_0\neq 0$ and
\begin{equation*}
    G(u,z):=\sum_{n=0}^\infty u_n \dfrac{z^n}{n!}<\infty,\quad z\in \mathds{C},\quad |z|\leq r,
\end{equation*}
for some $r>0$. Let $A(x)=(A_n(x))_{n\geq 0}$ be a sequence of polynomials such that $A(0)=(A_n(0))_{n\geq 0}\in \mathcal{G}$. Recall that $A(x)$ is called an Appell sequence if its generating functions has the form
\begin{equation*}
    G(A(x),z)=G(A(0),z)e^{xz}.
\end{equation*}
Denote by $\mathcal{A}$ the set of Appell sequences. Given $A(x), C(x)\in \mathcal{A}$, the binomial convolution of $A(x)$ and $C(x)$, denoted by $(A\times C)(x)=((A\times C)_n(x))_{n\geq 0}$, is defined as
\begin{equation*}
    (A\times C)_n(x)=\sum_{k=0}^n \binom{n}{k} A_k(0) C_{n-k}(x)=\sum_{k=0}^n \binom{n}{k}C_k(0) A_{n-k}(x).
\end{equation*}
It turns out (cf. \cite{AdeLek2017.2}) that $(A\times C)(x)$ is an Appell sequence characterized by its generating function
\begin{equation}\label{eq4.48}
    G((A\times C)(x),z)=G(A(0),z)G(C(0),z)e^{xz}.
\end{equation}
In fact, $(\mathcal{A},\times)$ is an abelian group with identity element $I(x)=(I_n(x))_{n\geq 0}$. Given $A(x)\in \mathcal{A}$, we consider the $k$-fold binomial convolution $A(k;x)=(A_n(k;x))_{n\geq 0}$ defined as
\begin{equation*}
    A_n(k;x)=(\stackrel{\stackrel{k}{\smile}}{A\times \cdots \times A})(x),\quad k\in \mathds{N},\quad A_n(0;x)=I_n(x).
\end{equation*}
As follows from \eqref{eq4.48}, $A(k;x)$ is an Appell sequence characterized by its generating function
\begin{equation}\label{eq4.49}
    G(A(k;x),z)=G^k(A(0),z)e^{xz},\quad k\in \mathds{N}_0.
\end{equation}
We mention the following examples. The generalized Bernoulli polynomials $B(k;x)=(B_n(k;x))_{n\geq 0}$ of integer order $k$ are defined by means of its generating function as
\begin{equation}\label{eq4.50}
    G(B(k;x),z))=\left ( \dfrac{z}{e^z-1} \right )^ke^{xz},\quad k\in \mathds{N}.
\end{equation}
From \eqref{eq4.49}, we see that $B(k;x)$ is the $k$-fold binomial convolution of the classical Bernoulli polynomials $B(x):=B(1;x)$. Similarly, the generalized Euler polynomials $E(k;x)=(E_n(k;x))_{n\geq 0}$ of integer order $k$ are defined via
\begin{equation}\label{eq4.51}
    G(E(k;x),z)=\left ( \dfrac{z}{e^z+1}\right )^k e^{xz},\quad k\in \mathds{N},
\end{equation}
and, therefore, $E(k;x)$ is the $k$-fold binomial convolution of the Euler polynomials $E(x):=E(1,x)$.

On the other hand, there are Appell sequences $A(x)$ which allow for a probabilistic representation of the form
\begin{equation}\label{eq4.52}
    A_n(x)=\mathds{E}(x+Y)^n,
\end{equation}
for a certain real or complex valued random variable $Y$. In terms of generating functions, \eqref{eq4.52} is equivalent to
\begin{equation}\label{eq4.53}
    G(A(x),z)=\sum_{n=0}^\infty \dfrac{A_n(x)}{n!}\, z^n=\mathds{E}e^{z(x+Y)}.
\end{equation}
For instance, the Bernoulli polynomials satisfy \eqref{eq4.52} with $Y=-1/2+i\zeta_1$, where $\zeta_1$ has the logistic distribution (see Sun \cite{Sun2005} or Ta \cite{Ta2015}) and $i$ is the imaginary unit, the Euler polynomials $E(x)$ fulfil \eqref{eq4.52} with $Y=i\zeta_2$, where $\zeta_2$ has the hyperbolic secant distribution (cf. Sun \cite{Sun2005} or Ta \cite{Ta2015}), and for the Hermite polynomials $H(x)$, \eqref{eq4.52} is true when $Y=iZ$, $Z$ having the standard normal distribution, as seen in \eqref{eq4.1.b}.

We are in a position to state the following result.

\begin{theorem}\label{th12}
Let $A(x)=(A_n(x))_{n\geq 0}$ be an Appell sequence satisfying \eqref{eq4.53}. For any $N\geq n$, we have
\begin{equation*}
    \sum_{k=0}^N A_n(k;x)= \sum_{k=0}^n c_{n,N}(k) A_n(k;x).
\end{equation*}
\end{theorem}

\begin{proof}
Using \eqref{eq4.49}, \eqref{eq4.53}, and the independence between the random variables involved, we have
\begin{equation*}
    G(A(k;x),z)=e^{xz} G^k (A(0),z)=\mathds{E}e^{z(x+S_k)},\quad k\in \mathds{N}_0,
\end{equation*}
thus implying that the $k$-fold binomial convolution $A(k;x)$ satisfies \eqref{eq4.53} when the random variable $Y$ is replaced by $S_k$. In other words,
\begin{equation*}
    A_n(k;x)=\mathds{E}(x+S_k)^n,\quad k\in \mathds{N}_0.
\end{equation*}
It therefore suffices to apply Corollary~\ref{c8} to complete the proof.
\end{proof}

In view of the comments preceding Theorem~\ref{th12}, this theorem applies when $A(x)$ is the sequence of Bernoulli, Euler and Hermite polynomials, among others. We finally mention that Kim and Kim \cite{KimKim2017} have recently obtained a closed form expression for certain sums of Appell sequences in terms of Barnes' multiple Bernoulli polynomials.






\section*{References}

\bibliographystyle{elsarticle-num}
\bibliography{AdellLekuonaArXiv2018}

\end{document}